\documentclass[letterpaper, 10 pt, conference]{ieeeconf}  

\usepackage{amsmath}
\usepackage{amsfonts}
\usepackage{cite} 
\usepackage{subfig}
\usepackage{graphicx}
\usepackage{siunitx}

\usepackage{amsthm}

\newtheoremstyle{brakets}
{}
{}
{\itshape}
{}
{}
{}
{\newline }
{\thmname{\bfseries {#1}}\thmnumber{ #2:}\thmnote{ \mdseries#3}}

\theoremstyle{brakets}
\newtheorem{theorem}{Theorem}

\newtheorem{lemma}[theorem]{Lemma}

\makeatletter
\renewenvironment{proof}[1][\proofname]{\par
	\pushQED{\qed}%
	\normalfont \topsep\z@
	\trivlist
	\item[\hskip2em
	\itshape
	#1\@addpunct{:}]\ignorespaces
}{%
	\popQED\endtrivlist\@endpefalse
}
\makeatletter

\IEEEoverridecommandlockouts                              

\title{\LARGE \bf
Feasibility of multiple robust control barrier functions\\ for bounding box constraints
}

\author{Mark Spiller,  Emilia Isbono, and Philipp Schitz
\thanks{*This work was supported by the German federal aviation research program (LuFo ID: 20D2111C).}
\thanks{The authors are with the
	German Aerospace Center (DLR), Institute of Flight Systems, 38108, Braunschweig, Germany. {\tt\small mark.spiller@dlr.de, emilia.isbono@dlr.de, philipp.schitz@dlr.de}}%
}

\begin{document}

\maketitle
\thispagestyle{empty}
\pagestyle{empty}

\begin{abstract}

Enforcing multiple constraints based on the concept of control barrier functions (CBFs) is a remaining challenge because each of the CBFs requires a condition on the control inputs to be satisfied which may easily lead to infeasibility problems. The problem becomes even more challenging with input constraints and disturbances. In this paper, we consider enforcement of bounding box constraints for a second order system under limited control authority and input disturbances. To solve the constrained control problem, we apply multiple robust control barrier functions (RCBFs) which, in general, do not provide a feasible solution to the problem. However, we derive conditions on how to select the RCBF parameters to guarantee that a feasible solution always exists. 

\end{abstract}

\section{Introduction}

In many fields of application (autonomous vehicles, robotics, aerospace systems, etc.), control requires constraints to be satisfied to enforce performance or safety requirements. Control barrier functions (CBFs) offer the ability to guarantee constraint satisfaction by rendering a subset of the state space forward invariant. An advantage of CBFs is that they can be applied in combination with a nominal controller, altering the control inputs in such a way that safety with respect to defined constraints is achieved \cite{ames2019control}. To make CBFs applicable to real-world systems, it is often necessary to additionally consider input constraints and disturbances. There exist contributions to handle the problem of CBF design under disturbances \cite{Alan2022Safe,Buch2022Robust,Cohen2022Robust,Seiler2022Control,Das2022Robust,Wang2023Robust,JANKOVIC2018359}, under input constraints \cite{Agrawal2021Safe,Breeden2021High,Dai2023Convex,Rabiee2023Soft,XIAO2022109960,Breeden2023Compositions,Zeng2021SafetyCritical}, or under input constraints and disturbances \cite{breeden2023robust,Choi2021Robust}. However, analyzing the feasibility of multiple CBFs, especially in presence of disturbances and input constraints, remains an open research topic. 

One practical way of handling multiple CBFs is prioritization \cite{Notomista2019An_O,Srinivasan2021ControlOf}. For this purpose, slack variables are introduced to relax the CBF conditions. Prioritization is then achieved by weighting of the slack variables. However, because of relaxation, constraint satisfaction cannot be guaranteed. Another way to handle the feasibility problem is to build one CBF from several CBFs which is then always feasible. In \cite{Black2023Adaptation}, a smooth candidate mutual CBF composed of individual CBFs and adaptive gains is formulated. If a quadratic program used to determine the adaptive gains is feasible, it can be guaranteed that the candidate mutual CBF is indeed a CBF. A mutual nonsmooth CBF via logic operators is constructed in \cite{Glotfelter2017Nonsmooth}. An extension of the logic operator concept with smooth CBFs is proposed in \cite{molnar2023composing}. Instead of a mutual CBF, the authors of \cite{Isaly2024OnTheFeas} introduce a vector-valued candidate CBF and study under which conditions the vector valued CBF can become valid. A grid-based sampling algorithm that examines feasibility of multiple CBFs is proposed in \cite{Tan2022Compatibility}. The work of \cite{Breeden2023Compositions} studies mutual feasibility of CBFs for a second-order system. A drawback of all aforementioned approaches is that they do not explicitly consider input constraints and disturbances to be present. 

This paper considers the problem of satisfying bounding box constraints for a second order system under limited control inputs and input disturbances. Using the concept of robust control barrier functions (RCBFs), a single constraint can be enforced in presence of disturbances if sufficient control authority is available \cite{breeden2023robust}. However, for the bounding box constraint considered in this paper, RCBFs do not guarantee a feasible solution in general. We derive conditions on how to select the RCBF parameters so that a solution to the constrained control problem can always be found.

The paper is organized as follows. In Section \ref{sec:background}, a brief review about the concept of RCBFs is given. The feasibility of RCBFs for bounding box constraints is studied in Section \ref{sec:main_results}. A numerical example of altering the flight level of an aerial vehicle under input and state constraints as well as disturbances is considered in Section \ref{sec:num_example}.

\section{Background}
\label{sec:background}

Consider an input-affine nonlinear system
\begin{align}
	\dot x
	=
	f(x)+g(x)u+g(x)w
\end{align}
with states 
$x \in \mathbb{R}^n$, constrained control inputs $u\in \mathcal{U}\subset \mathbb{R}^m$ and input disturbances $w\in \mathcal{W}\subset \mathbb{R}^m$, where $\mathcal{U}$ and $\mathcal{W}$ are compact. Functions $f(x)\colon \mathbb{R}^n \to \mathbb{R}^n$ and $g(x)\colon \mathbb{R}^n \to \mathbb{R}^{n\times m}$ are assumed locally Lipschitz continuous in $x$. The states are constrained by a constraint function	$h(t,x)\colon [t_0,\infty)\times \mathbb{R}^n \to \mathbb{R}$ with $h\in \mathcal{G}^2$, where $\mathcal{G}^2$ denotes the set of functions that are $2$-times total differentiable in time and are of relative degree two in the sense that $u$ and $w$ appear explicitly in the second time derivative. Based on the constraint function $h$ the time-varying zero sublevel set 
\begin{align}
	\mathcal{S}(t)=\{x \in \mathbb{R}^n \mid h(t,x)\le0\}
\end{align}
is defined.
The goal is to achieve forward invariance of a subset of $\mathcal{S}(t)$ in order to satisfy the constraints described by $h$ under input constraints and disturbances. This problem is studied in detail in \cite{breeden2023robust} where a solution is provided by so-called robust control barrier functions (RCBFs). The definition of RCBFs is given as follows based on Theorem 9 of \cite{breeden2023robust}.

\begin{theorem}[{\cite[Theorem~9]{breeden2023robust}}]
Suppose $h\in\mathcal{G}^2$ and there exists $a_{\max}>0$ such that $\forall (t,x)\in [t_0,\infty)\times \mathcal{S}(t)\colon$
\begin{align}
	\max_{w\in \mathcal{W}}
	\left(
	\inf_{u\in \mathcal{U}}
	\ddot h(t,x,u,w)
	\right)
	\le -a_{\max}.
\end{align}
Then the function
\begin{align}
	H(t,x) = h(t,x) + \frac{|\dot h(t,x)|\dot h(t,x)}{2a_{\max}}
\end{align}
is a RCBF on the set
\begin{align}
	\mathcal{S}^{res}_H(t)
	=
	\{x \in \mathbb{R}^n \mid 
	H(t,x)\le0
	\land
	h(t,x)\le0
	\}
	\subseteq
	\mathcal{S}(t)\notag,
\end{align}
i.e.
\begin{align}
	\max_{w\in \mathcal{W}}\left [	\inf_{u\in \mathcal{U}}
	\left (\dot H(t,x,u,w)\right)
	\right]
	\le
	\alpha (-H(t,x))
\end{align}
for any class-$\mathcal{K}$ function $\alpha$ and any $t\in[t_0,\infty)$, $x\in\mathcal{S}^{res}_H(t)$. The set
\begin{align}
	&
	U(t,x)=
	\left\{
	u\in\mathcal{U}
	\mid
	\max_{w\in \mathcal{W}}\left (	
	\dot H(t,x,u,w)
	\right)
	\le
	\alpha (-H(t,x))
	\right\}
	\notag
\end{align}
is non-empty for any class-$\mathcal{K}$ function $\alpha$ and any $t\in[t_0,\infty)$, $x\in\mathcal{S}^{res}_H(t)$. Moreover, any control law $u(t,x)$ that is locally Lipschitz continuous in $x$ and piecewise continuous in $t$ such that $\forall (t,x)\in[t_0,\infty)\times\mathcal{S}^{res}_H(t) \colon u(t,x) \in U(t,x)$ renders $\mathcal{S}^{res}_H$ forward invariant.
\label{num:theorem_theo9_rcbf_paper}
\end{theorem}

\section{Main Results}
\label{sec:main_results}

Consider the second-order system
\begin{align}
	\begin{bmatrix}
		\dot x_1\\
		\dot x_2
	\end{bmatrix}
	=
	\begin{bmatrix}
		x_2\\
		f(x_1,x_2)+ g(x_1,x_2)(\mu + \nu)
	\end{bmatrix}
	\label{num:mr_spec_sys_desc_pre}
\end{align}
with states 
$[x_1, x_2]^T \in \mathbb{R}^2$, control input $\mu\in[\mu_{lb},\mu_{ub}]\subset  \mathbb{R}$, input disturbance $\nu\in [-\nu_{\max},\nu_{\max}]\subset \mathbb{R}$, and $\forall x_1, x_2\colon g(x_1,x_2)\not=0, |g^{-1}(x_1,x_2)|\le M$. State $x_1$ is assumed to be constrained as
\begin{align}
	h_{ub}(t,x) &= x_1-l_{ub}(t)\le0,
	\label{num:mr_def_ub_cstr}\\
	h_{lb}(t,x) &= -x_1+l_{lb}(t)\le0
	\label{num:mr_def_lb_cstr}
\end{align}
with time-varying bounds $l_{ub}(t)\ge l_{lb}(t)$. For the bounds, we assume
$\dot l_{ub}(t)=\dot l_{lb}(t)=\dot l(t)$
as well as  $\max_t\{|\ddot l(t)|\}\le \ddot l_{\max}$. We apply $\mu=g^{-1}(u-f)$ to transform (\ref{num:mr_spec_sys_desc_pre}) into
\begin{align}
	\begin{bmatrix}
		\dot x_1\\
		\dot x_2
	\end{bmatrix}
	=
	\begin{bmatrix}
		x_2\\
		u + g(x_1,x_2)\nu
	\end{bmatrix}
	=
	\begin{bmatrix}
		x_2\\
		u + w
	\end{bmatrix}
	\label{num:mr_spec_sys_desc}
\end{align}
with control input $u\in [u_{lb},u_{ub}]\subset \mathbb{R}$ and disturbance $w\in [-w_{\max},w_{\max}]\subset \mathbb{R}$. From (\ref{num:mr_def_ub_cstr}), (\ref{num:mr_def_lb_cstr}), (\ref{num:mr_spec_sys_desc}) it follows
\begin{alignat}{3}
	\dot h_{ub} &= x_2- \dot l
	\qquad
	&\dot h_{lb} &= -x_2+\dot l,
	\label{num:mr_def_h_ub_h_lb_dot}\\
	\ddot h_{ub} &= u + w- \ddot l
	\qquad
	&\ddot h_{lb} &= -u - w+\ddot l.
	\label{num:mr_def_h_ub_h_lb_ddot}
\end{alignat}
We assume that the available control authority $u\in [u_{lb},u_{ub}]$ is sufficient to achieve
\begin{align}
	&
	\max\{
	u_{lb} + w_{\max}+ \ddot l_{\max},
	-u_{ub} +w_{\max} +\ddot l_{\max}
	\}
	\notag\\
	&\qquad \qquad \qquad \qquad \qquad \qquad \qquad \quad
	\le - a_{\max}
	<0
	\label{num:mr_asspt_ul_ub_cstr}
\end{align} 
so that 
\begin{align}
	&\max_{|w|\le w_{\max}}
	\left(
	\inf_{u\in [u_{lb},u_{ub}]}
	\ddot h_{ub}
	\right)
	= 
	u_{lb} + w_{\max}- \ddot l\notag\\
	&
	\qquad \qquad \qquad \quad\le
	u_{lb} + w_{\max}+ \ddot l_{\max}
	\le - a_{\max}
	,\\
	&\max_{|w|\le w_{\max}}
	\left(
	\inf_{u\in [u_{lb},u_{ub}]}
	\ddot h_{lb}
	\right)
	= 
	-u_{ub} +w_{\max} +\ddot l\notag\\
	&
	\qquad \qquad \qquad \quad \le
	-u_{ub} +w_{\max} +\ddot l_{\max}
	\le - a_{\max} 
\end{align} 
are guaranteed to hold for some $a_{\max}>0$. In addition, we define
\begin{align}
	H_{ub}(t,x)
	&=
	h_{ub}(t,x)
	+
	\frac{|\dot h_{ub}(t,x)|\dot h_{ub}(t,x)}{2a_{\max}}
	\notag\\
	&=
	x_1-l_{ub}
	+
	\frac{|x_2- \dot l|(x_2- \dot l)}{2a_{\max}},
	\label{num:mr_def_capH_ub}\\
	H_{lb}(t,x)
	&=
	h_{lb}(t,x)
	+
	\frac{|\dot h_{lb}(t,x)|\dot h_{lb}(t,x)}{2a_{\max}}
	\notag\\
	&=
	-x_1+l_{lb}
	-
	\frac{|x_2-\dot l|(x_2-\dot l)}{2a_{\max}}
	\label{num:mr_def_capH_lb}
\end{align}
and the sets
\begin{align}
	\mathcal{S}^{res}_{ub}(t)
	&=
	\left\{x \in \mathbb{R}^2 \mid 
	H_{ub}(t,x)\le0
	\enspace
	\land
	\enspace
	h_{ub}(t,x)\le0
	\right \},\notag\\
	\mathcal{S}^{res}_{lb}(t)
	&=
	\left\{x \in \mathbb{R}^2 \mid 
	H_{lb}(t,x)\le0 
	\enspace
	\land
	\enspace
	h_{lb}(t,x)\le0
	\right \}.\notag
\end{align}
The goal is to render the set $\mathcal{S}^{res}_c(t)$ defined as
\begin{align}
	& \mathcal{S}^{res}_c(t)
	=
	\mathcal{S}^{res}_{ub}(t)
	\cap
	\mathcal{S}^{res}_{lb}(t)\notag\\
	& \qquad
	\subseteq
	\left\{x \in \mathbb{R}^2 \mid 
	h_{ub}(t,x)\le0
	\enspace
	\land
	\enspace
	h_{lb}(t,x)\le0
	\right \}
	\label{num:mr_intersection_sets}
\end{align}
forward invariant under input constraints and input disturbances. 
\begin{lemma}[Forward invariance of $\mathcal{S}^{res}_c(t)$]
	Consider the set
	\begin{align}
		&
		U_c(t,x)
		\in
		\{
		u\in [u_{lb},u_{ub}]
		\mid
		\notag\\
		& \qquad \qquad
		g_{lb}(t,x,u,w)
		\le0
		\land 
		g_{ub}(t,x,u,w)
		\le0\}
	\end{align}
	with
	\begin{align}
		&
		g_{ub}(t,x,u)
		=
		\max_{|w|\le w_{\max}}
		\left (	
		\dot H_{ub}(t,x,u,w) 
		\right) \notag\\
		& \qquad \qquad\qquad \qquad \qquad \qquad
		-
		\alpha_{ub} 
		\left(
		-H_{ub}(t,x) 
		\right)
		,
		\label{num:mr_def_rcbf_str_ub}\\
		&
		g_{lb}(t,x,u)
		=
		\max_{|w|\le w_{\max}}
		\left (	
		\dot H_{lb} (t,x,u,w)
		\right) \notag\\
		& \qquad \qquad\qquad \qquad \qquad \qquad
		-
		\alpha_{lb} 
		\left(
		-H_{lb}(t,x)
		\right).
		\label{num:mr_def_rcbf_str_lb}
	\end{align}
	Any control law $u(t,x)$ that is locally Lipschitz continuous in $x$ and piecewise continuous in $t$ such that $\forall (t,x)\in[t_0,\infty)\times\mathcal{S}_c^{res}(t) \colon u(t,x) \in U_c(t,x)$ renders $\mathcal{S}^{res}_c(t)$ forward invariant under input constraints $u\in [u_{lb},u_{ub}]$ and disturbances $|w|\le w_{\max}$.
	\label{num:Leamma_forward_inv_Sc}
\end{lemma}
\begin{proof}
	As $u(t,x)$ is Lipschitz continuous and satisfies (\ref{num:mr_def_rcbf_str_ub}) for all $t\in[t_0,\infty)$, $x\in\mathcal{S}_{ub}^{res}(t)$, it follows from Theorem \ref{num:theorem_theo9_rcbf_paper} that $\mathcal{S}_{ub}^{res}(t)$ is rendered forward invariant. Similarly, as $u(t,x)$ satisfies (\ref{num:mr_def_rcbf_str_lb}) for all $t\in[t_0,\infty)$, $x\in\mathcal{S}_{lb}^{res}(t)$, it follows from Theorem \ref{num:theorem_theo9_rcbf_paper} that $\mathcal{S}_{lb}^{res}(t)$ is rendered forward invariant. Hence, forward invariance of $\mathcal{S}^{res}_c(t)$ follows from $\mathcal{S}^{res}_c(t)
	=
	\mathcal{S}^{res}_{ub}(t)
	\cap
	\mathcal{S}^{res}_{lb}(t)$.
\end{proof}	
Before proceeding with the feasibility analysis, we study the boundedness of state $x_2$.
\begin{lemma}[Boundedness of $|x_2-\dot l|$]
	If $H_{ub}\le0$, $H_{lb}\le0$, $h_{ub}\le0$, and $h_{lb}\le0$, then 
	\begin{align}
		|x_2-\dot l|
		&
		\le
		\sqrt{(l_{ub}(t_0)-l_{lb}(t_0))2a_{\max}}
		\label{num:bdn_x2_1}
	\end{align}
	assuming $l_{ub}(t_0)-l_{lb}(t_0)\ge0$.
	\label{num:lemma_bdn_x_2}
\end{lemma}
\begin{proof}
	From (\ref{num:mr_def_capH_ub}) it follows
\begin{align}
	&
	|x_2- \dot l|(x_2- \dot l)\le (-x_1+l_{ub})2a_{\max}
\end{align}
for $H_{ub}\le0$. Substituting (\ref{num:mr_def_lb_cstr}) yields
\begin{align}
	&
	|x_2- \dot l|(x_2- \dot l)
	\le(l_{ub}-l_{lb})2a_{\max}.
	\label{num:bdn_x2_proof_1}
\end{align}
Similarly, from $H_{lb}\le0$, $h_{ub}\le0$ and (\ref{num:mr_def_capH_lb}), (\ref{num:mr_def_ub_cstr}) it follows
\begin{align}
	&
	-|x_2-\dot l|(x_2-\dot l)
	\le
	(l_{ub}-l_{lb})2a_{\max}.
	\label{num:bdn_x2_proof_2}
\end{align}
Evaluating (\ref{num:bdn_x2_proof_1}) for $x_2- \dot l\ge0$ yields
\begin{align}
	&
	0
	\le
	x_2 - \dot l,
	\qquad
	x_2 - \dot l
	\le 
	\sqrt{(l_{ub}-l_{lb})2a_{\max}}
\label{num:bdn_x2_proof_3}
\end{align}
and evaluating (\ref{num:bdn_x2_proof_2}) for $x_2- \dot l\le0$ yields
\begin{align}
	&
	-
	\sqrt{(l_{ub}-l_{lb})2a_{\max}}
\le
x_2 - \dot l	
,
\qquad
	x_2
	-\dot l
\le
0.
\label{num:bdn_x2_proof_4}
\end{align}
Because of $\dot l_{ub}=\dot l_{lb}=\dot l$, it follows $l_{ub}(t) - l_{lb}(t)=l_{ub}(t_0) - l_{lb}(t_0)$, then from (\ref{num:bdn_x2_proof_3}) and (\ref{num:bdn_x2_proof_4}), we derive (\ref{num:bdn_x2_1}).
\end{proof}

The main result with regard to feasibility can now be stated.
\begin{theorem}[Feasibility conditions]
	If $x(t_0)
			\in
			\mathcal{S}^{res}_c(t_0)$
	and $\alpha_{ub}, \alpha_{lb}\in \mathbb{R}_{>0}$ satisfy
				\begin{align}
			&
	\alpha_{ub}
	=
	\alpha_{lb}
	\ge
		2w_{\max}
		\sqrt{
		\frac{2}
		{(l_{ub}(t_0)-l_{lb}(t_0))a_{\max}}
		}	
		\label{num:feas_safeset_1aa}
	\end{align}
	with $l_{ub}(t_0)-l_{lb}(t_0)>0$,
	then for all $t\in[t_0,\infty)$ and $x\in\mathcal{S}^{res}_c(t)$ the set
		\begin{align}
		&
		U_c(t,x)
		=
		\left\{
		u\in[u_{lb},u_{ub}]
		\mid
		\Phi_{lb}(t,x)
		\le
		u
		\le
		\Phi_{ub}(t,x)
		\right\}
		\label{num:thoe_feas_cond_Umutual}
	\end{align}
		where
	\begin{align}
		\Phi_{lb}(t,x)
		&=	
		-
		\frac{a_{\max}}{|x_2-\dot l|}
		\alpha_{lb}
		\left(
		x_1
		-l_{lb}
		+
		\frac{|x_2-\dot l|(x_2-\dot l)}{2a_{\max}}
		\right)\notag\\
		&\qquad \quad 
		-
		\emph{\text{sgn}}(x_2-\dot l)a_{\max}
		+
		w_{\max}
		+
		\ddot l,
		\label{num:thoe_feas_cond_U_lb}\\
		\Phi_{ub}(t,x)
		&=
		\frac{a_{\max}}{|x_2-\dot l|}
		\alpha_{ub}
		\left(
		-x_1+l_{ub}
		-
		\frac{|x_2- \dot l|(x_2- \dot l)}{2a_{\max}}
		\right)\notag\\
		& \qquad \quad 
		-\emph{\text{sgn}}(
		x_2
		-
		\dot l)a_{\max}
		-
		w_{\max}
		+
		\ddot l
		\label{num:thoe_feas_cond_U_ub}
	\end{align}
	is guaranteed to be non-empty.
	Moreover, any control law $u(t,x)$ that is locally Lipschitz continuous in $x$ and piecewise continuous in $t$ such that $\forall(t,x) \in [t_0,\infty) \times \mathcal{S}^{res}_c(t) \colon u(t,x)\in U_c(t,x)$	renders $\mathcal{S}^{res}_c(t)$ forward invariant under input constraints $u\in [u_{lb},u_{ub}]$ and disturbances $|w|\le w_{\max}$.
	\label{num:theorem_feas_safeset}
\end{theorem}
\begin{proof}
	In the following proof we require linear class-$\mathcal{K}$ functions and therefore make the restriction $\alpha_{ub}, \alpha_{lb}\in \mathbb{R}_{>0}$. According to Lemma \ref{num:Leamma_forward_inv_Sc}, to achieve forward invariance of $\mathcal{S}^{res}_c$, it is required that a control input $u\in[u_{lb},u_{ub}]$ exists that satisfies $g_{ub}\le0$ and $g_{lb}\le0$.
	
	\subsection{Upper bound RCBF constraint}
	\label{num:subsectheo_rcbf_ub}
	The upper bound RCBF constraint is $g_{ub}\le0$ with $g_{ub}$ from (\ref{num:mr_def_rcbf_str_ub}). Taking the time-derivative of (\ref{num:mr_def_capH_ub}) and substituting (\ref{num:mr_def_h_ub_h_lb_dot}) and (\ref{num:mr_def_h_ub_h_lb_ddot}) leads to
	\begin{align}
		\dot H_{ub}
		&=
		\dot h_{ub}
		+
		\frac{|\dot h_{ub}|\ddot h_{ub}}{a_{\max}}
		=
		x_2-\dot l
		+
		\frac{|x_2-\dot l|(u+w-\ddot l)}{a_{\max}}\notag
	\end{align}
	so that
	\begin{align}
		&
		\max_{|w|\le w_{\max}}
		 (	
		\dot H_{ub}
		)
		=
		x_2
		-
		\dot l
		+
		\frac{|x_2-\dot l|(u-\ddot l)}{a_{\max}}
		+
		\frac{|x_2-\dot l|w_{\max}}{a_{\max}}.
		\label{num:feas_safeset_proof_3}
	\end{align}
	Substituting (\ref{num:feas_safeset_proof_3}) and (\ref{num:mr_def_capH_ub}) in $g_{ub}\le0$ and rearranging yields 
	\begin{align}
		&
		\Phi_{ub}(t,x)
		=\frac{a_{\max}}{|x_2-\dot l|}
		\alpha_{ub}
		\left(
		-x_1+l_{ub}
		-
		\frac{|x_2- \dot l|(x_2- \dot l)}{2a_{\max}}
		\right)\notag\\
		& \qquad \qquad \qquad
		-\text{sgn}(
		x_2
		-
		\dot l)a_{\max}
		-
		w_{\max}
		+
		\ddot l
		\ge u
		\label{num:feas_safeset_proof_5}
	\end{align}
	so that $g_{ub}=-\Phi_{ub}+u\le0$ is satisfied if (\ref{num:feas_safeset_proof_5}) is satisfied.
	
	\subsection{Lower bound RCBF constraint}
	\label{num:subsectheo_rcbf_lb}
	The lower bound RCBF constraint is $g_{lb}\le0$ with $g_{lb}$ from (\ref{num:mr_def_rcbf_str_lb}).
	Taking the time-derivative of (\ref{num:mr_def_capH_lb}) and substituting (\ref{num:mr_def_h_ub_h_lb_dot}) and (\ref{num:mr_def_h_ub_h_lb_ddot}) leads to
	\begin{align}
		\dot H_{lb}
		&=
		\dot h_{lb}
		+
		\frac{|\dot h_{lb}|\ddot h_{lb}}{a_{\max}}
		=
		-x_2+\dot l
		+
		\frac{|-x_2+\dot l|(-u - w+\ddot l)}{a_{\max}}\notag
	\end{align}
	so that
	\begin{align}
		\max_{|w|\le w_{\max}}
		(	
		\dot H_{lb}
		)
		&=
		\dot l
		-x_2
		+
		\frac{|x_2-\dot l|(\ddot l-u)}{a_{\max}}
		+
		\frac{|x_2-\dot l|w_{\max}}{a_{\max}}.
		\label{num:feas_safeset_proof_8}
	\end{align}
	Substituting (\ref{num:feas_safeset_proof_8}) and (\ref{num:mr_def_capH_lb}) in $g_{lb}\le0$ yields
	\begin{align}
	u
	&\ge
	-
	\frac{a_{\max}}{|x_2-\dot l|}
	\alpha_{lb}
	\left(
	x_1
	-l_{lb}
	+
	\frac{|x_2-\dot l|(x_2-\dot l)}{2a_{\max}}
	\right)\notag\\
	&\qquad
	-
	\text{sgn}(x_2-\dot l)a_{\max}
	+
	w_{\max}
	+
	\ddot l
	= \Phi_{lb}(t,x)
	\label{num:feas_safeset_proof_10}
	\end{align}
	so that $g_{lb}=\Phi_{lb}-u\le0$ is satisfied if (\ref{num:feas_safeset_proof_10}) is satisfied.
	
	\subsection{Feasibility of mutual RCBF constraint}
	According to (\ref{num:feas_safeset_proof_5}) and (\ref{num:feas_safeset_proof_10}) the mutual RCBF constraint $\Phi_{lb}(t,x)
	\le
	u
	\le
	\Phi_{ub}(t,x)$ is feasible with respect to $u$ if
	\begin{align}
			\Phi_{ub}(t,x)
			&=
			\frac{a_{\max}}{|x_2-\dot l|}
			\alpha_{ub}
			\left(
			-x_1+l_{ub}
			-
			\frac{|x_2- \dot l|(x_2- \dot l)}{2a_{\max}}
			\right)
			\notag\\
			& \quad 
			-\text{sgn}(
			x_2
			-
			\dot l)a_{\max}
			-
			w_{\max}
			+
			\ddot l
			\notag\\
			& 
			\ge 
			-
			\frac{a_{\max}}{|x_2-\dot l|}
			\alpha_{lb}
			\left(
			x_1
			-l_{lb}
			+
			\frac{|x_2-\dot l|(x_2-\dot l)}{2a_{\max}}
			\right)
			\notag\\
			& \quad
			-
			\text{sgn}(x_2-\dot l)a_{\max}
			+
			w_{\max}
			+
			\ddot l
			= \Phi_{lb}(t,x)
		\label{num:feas_safeset_proof_11}
	\end{align}
	holds. If the states are in $\mathcal{S}^{res}_c(t)$, then $H_{ub}\le0$ and $H_{lb}\le0$ hold. Therefore, it follows
	\begin{align}
		&
		-H_{ub}
		=
		-x_1+l_{ub}
		-
		\frac{|x_2- \dot l|(x_2- \dot l)}{2a_{\max}}
		\ge0,
		\label{num:feas_safeset_proof_13}\\
		&
		-H_{lb}
		=
		x_1
		-
		l_{lb}
		+
		\frac{|x_2-\dot l|(x_2-\dot l)}{2a_{\max}}
		\ge0
		\label{num:feas_safeset_proof_14}.
	\end{align}
	By introducing $\alpha=\min\{\alpha_{lb},\alpha_{ub}\}$, implying $\alpha_{ub}\ge\alpha$ and $-\alpha\ge-\alpha_{lb}$, it follows from
	(\ref{num:feas_safeset_proof_11}), (\ref{num:feas_safeset_proof_13}), (\ref{num:feas_safeset_proof_14}) that mutual feasibility of the RCBF constraints can be achieved if
		\begin{align}
			&\frac{a_{\max}}{|x_2-\dot l|}
		\alpha
		\left(
		-x_1+l_{ub}
		-
		\frac{|x_2- \dot l|(x_2- \dot l)}{2a_{\max}}
		\right)\notag\\
		& \qquad \qquad \qquad
		-\text{sgn}(
		x_2
		-
		\dot l)a_{\max}
		-
		w_{\max}
		+
		\ddot l
		\notag\\
		& \ge
		-
		\frac{a_{\max}}{|x_2-\dot l|}
		\alpha
		\left(
		x_1
		-l_{lb}
		+
		\frac{|x_2-\dot l|(x_2-\dot l)}{2a_{\max}}
		\right)\notag\\
		&\qquad \qquad \qquad
		-
		\text{sgn}(x_2-\dot l)a_{\max}
		+
		w_{\max}
		+
		\ddot l
		\label{num:feas_safeset_proof_15}
	\end{align}
	holds. Rearranging (\ref{num:feas_safeset_proof_15})
	yields
		\begin{align}
		&
		\alpha
		\ge
		2w_{\max}
		\frac{|x_2-\dot l|}{a_{\max}(	l_{ub}
			-l_{lb})}
		\label{num:feas_safeset_proof_16b}
	\end{align}
	for $l_{ub}
	-l_{lb}>0$. Considering $l_{ub}(t) - l_{lb}(t)=l_{ub}(t_0) - l_{lb}(t_0)$ and Lemma \ref{num:lemma_bdn_x_2},  (\ref{num:feas_safeset_1aa}) follows from (\ref{num:feas_safeset_proof_16b}). The restriction $\alpha =\alpha_{lb}=\alpha_{ub}$ results from the input constraints as shown below.

\subsection{Input constraints}
It remains to show that at least one of the feasible solutions of the mutual RCBF constraint also satisfies the input constraints. With (\ref{num:thoe_feas_cond_U_lb}), (\ref{num:thoe_feas_cond_U_ub}), (\ref{num:feas_safeset_proof_13}), and
(\ref{num:feas_safeset_proof_14}) the mutual RCBF constraint $\Phi_{lb}(t,x)
\le
u
\le
\Phi_{ub}(t,x)$ is
	\begin{align}
	&
	\frac{a_{\max}}{|x_2-\dot l|}
	\alpha_{ub}
	(-H_{ub})
	-\text{sgn}(
	x_2
	-
	\dot l)a_{\max}
	-
	w_{\max}
	+
	\ddot l
	\ge
	u\notag\\
	\ge
	&
	-
	\frac{a_{\max}}{|x_2-\dot l|}
	\alpha_{lb}
	(-H_{lb})
	-
	\text{sgn}(x_2-\dot l)a_{\max}
	+
	w_{\max}
	+
	\ddot l
	.
		\label{num:feas_safeset_proof_23}
\end{align}
From (\ref{num:mr_def_capH_ub}), (\ref{num:mr_def_capH_lb}) and $-H_{ub}\ge0$ and $-H_{lb}\ge0$ it follows
\begin{align}
	H_{ub}
	&=
	-H_{lb}
	+l_{lb}
	-l_{ub}
	\ge
	l_{lb}
	-l_{ub}
	,\label{num:feas_safeset_proof_24}\\
	H_{lb}
	&=
	-H_{ub}
	+l_{lb}
	-l_{ub}
	\ge
	l_{lb}
	-l_{ub}
	\label{num:feas_safeset_proof_25}
\end{align}
which with (\ref{num:feas_safeset_proof_23}) yields
	\begin{align}
	&
	\frac{a_{\max}}{|x_2-\dot l|}
	\alpha_{ub}
	(l_{ub}-l_{lb})
	-\text{sgn}(
	x_2
	-
	\dot l)a_{\max}
	-
	w_{\max}
	+
	\ddot l
	\notag\\
	&\ge
	\Phi_{ub}(t,x)
	\ge
	u
	\ge
	\Phi_{lb}(t,x)
		\ge
	-
	\frac{a_{\max}}{|x_2-\dot l|}
	\alpha_{lb}
	(l_{ub}-l_{lb})
	\notag\\
	&
	\qquad  
	-
	\text{sgn}(x_2-\dot l)a_{\max}
	+
	w_{\max}
	+
	\ddot l
		\label{num:feas_safeset_proof_26}.
\end{align}
According to (\ref{num:feas_safeset_proof_16b}), feasible solutions for $\Phi_{ub}
\ge
u
\ge
\Phi_{lb}$ with respect to $u$ exist if $l_{ub}
-
l_{lb}
\ge
2
w_{\max}
\frac{|x_2-\dot l|}{a_{\max}\alpha}$. Therefore, by substituting $2
w_{\max}
\frac{|x_2-\dot l|}{a_{\max}\alpha}$ for $l_{ub}
-
l_{lb}$  in (\ref{num:feas_safeset_proof_26}), there always exists at least one feasible solution within the bounds
	\begin{align}
	&
2
w_{\max}
\frac{\alpha_{ub}}{\alpha}
-\text{sgn}(
x_2
-
\dot l)a_{\max}
-
w_{\max}
+
\ddot l
\ge
u\notag\\
\ge
&
-
2
w_{\max}
\frac{\alpha_{lb}}{\alpha}
-
\text{sgn}(x_2-\dot l)a_{\max}
+
w_{\max}
+
\ddot l
	\label{num:feas_safeset_proof_27}
\end{align}
of $u$.
For the upper and lower bounds of the input constraints it follows 
	\begin{align}
	u_{ub}
	&\ge
	a_{\max}
	+
	w_{\max}
	+
	\ddot l_{\max},
	\label{num:feas_safeset_proof_28}\\
	u_{lb}
	&\le
	-a_{\max}
	-
	w_{\max}
	-
	\ddot l_{\max}
	\label{num:feas_safeset_proof_29}
\end{align}
from (\ref{num:mr_asspt_ul_ub_cstr}). Considering the upper and lower bounds of (\ref{num:feas_safeset_proof_27}) for $\alpha_{ub}=\alpha_{lb}=\alpha$ and substituting (\ref{num:feas_safeset_proof_28}) and (\ref{num:feas_safeset_proof_29}) in (\ref{num:feas_safeset_proof_27}) shows that the feasible solution lies within the bounds of the input constraints.

\subsection{Forward invariance of $\mathcal{S}^{res}_c(t)$}
The constraints $\Phi_{lb}
\le
u
\le
\Phi_{ub}$ and $g_{lb}=\Phi_{lb}-u
\le0
\land 
g_{ub}=-\Phi_{ub}+u\le0$ are equivalent. Therefore, by applying any Lipschitz continuous controller $u(t,x)\in U_c(t,x)$ with $U_c(t,x)$ of (\ref{num:thoe_feas_cond_Umutual}) it follows from Lemma \ref{num:Leamma_forward_inv_Sc} that $\mathcal{S}^{res}_c(t)$ is rendered forward invariant.
\end{proof}

\section{Numerical Example}	
\label{sec:num_example}

\begin{figure*}
	\subfloat[]{%
		\includegraphics[width=0.49\linewidth]{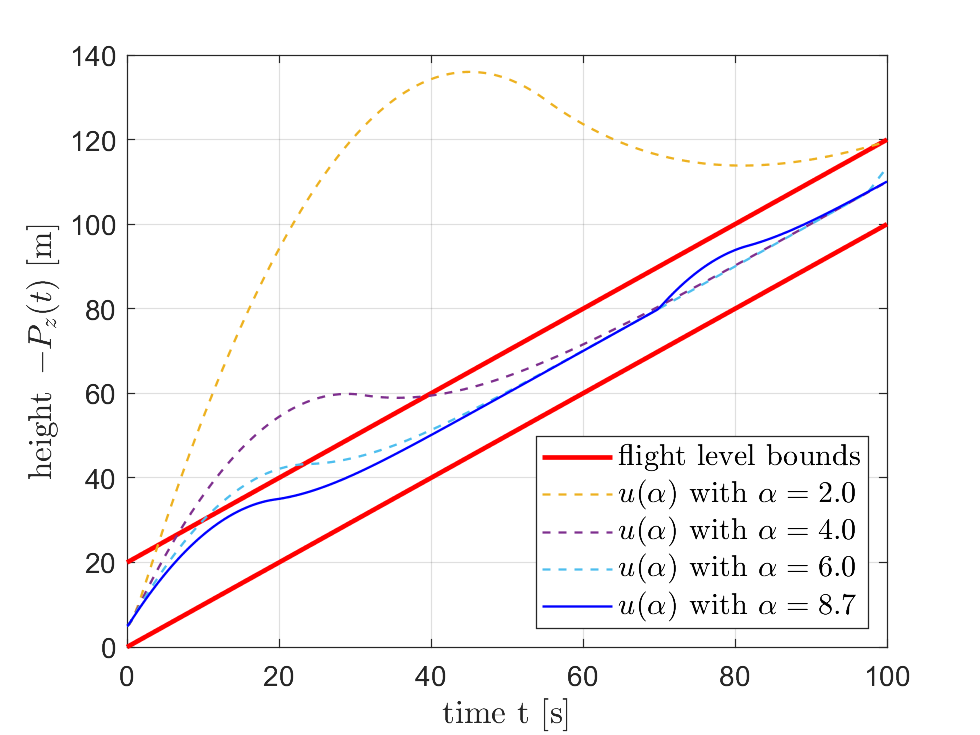}}
	\hfill
	\subfloat[]{\includegraphics[width=0.49\linewidth]{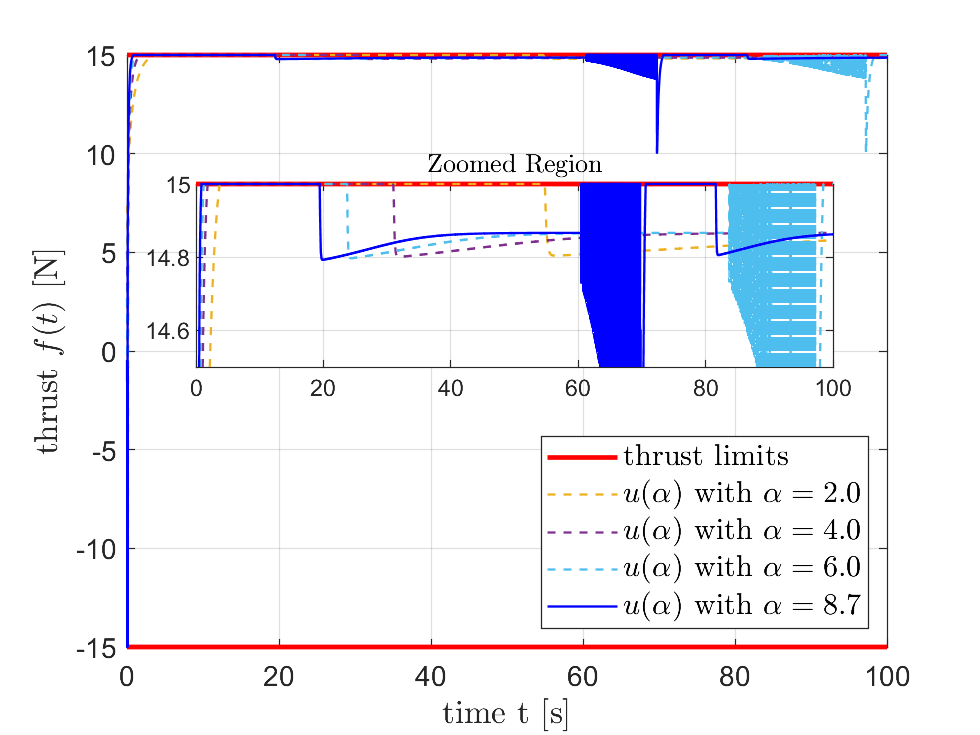}}
	\vspace*{-0.3cm}\\
	\subfloat[]{
		\includegraphics[width=0.49\linewidth]{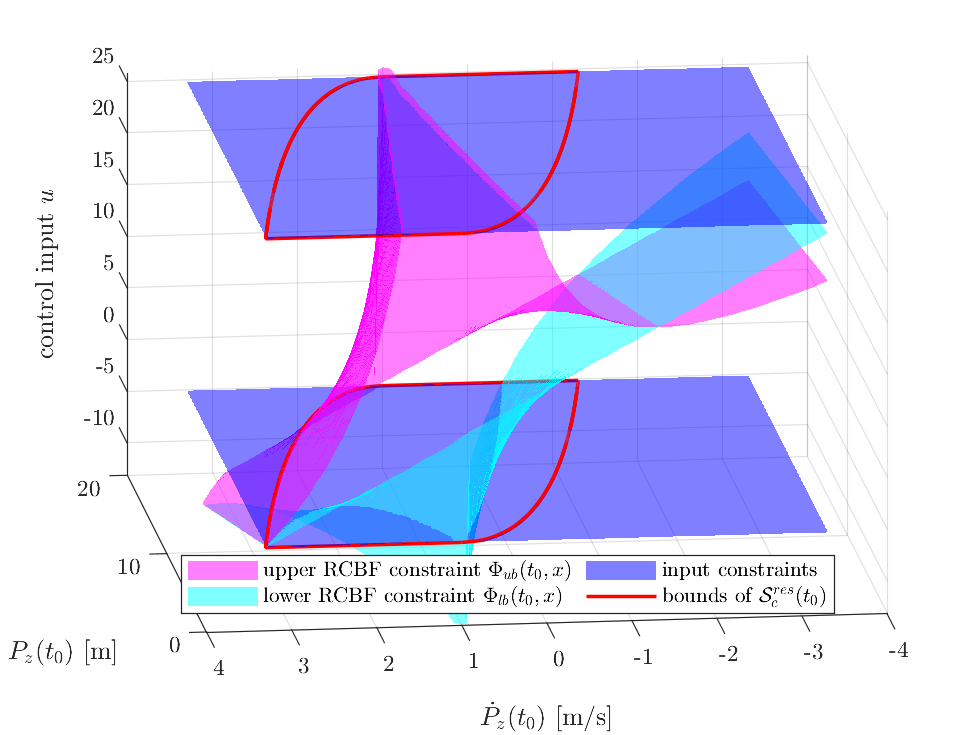}}
	\hfill
	\subfloat[]{\includegraphics[width=0.49\linewidth]{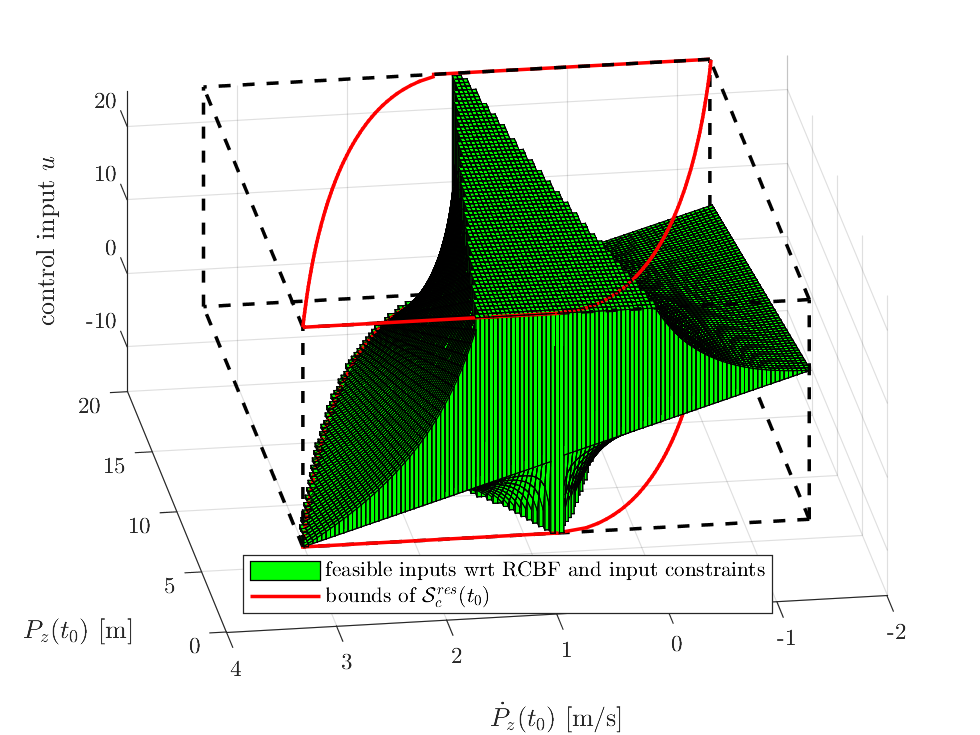}}
	\vspace*{-0.3cm}\\
	\subfloat[]{\includegraphics[width=0.49\linewidth]{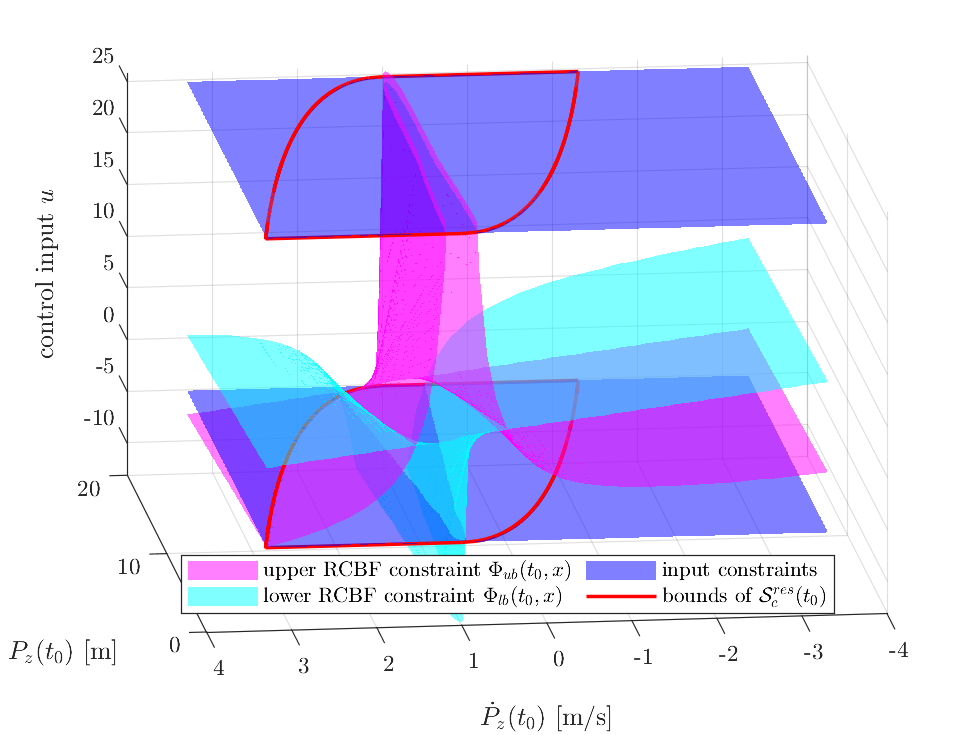}}
	\hfill
	\subfloat[]{%
		\includegraphics[width=0.49\linewidth]{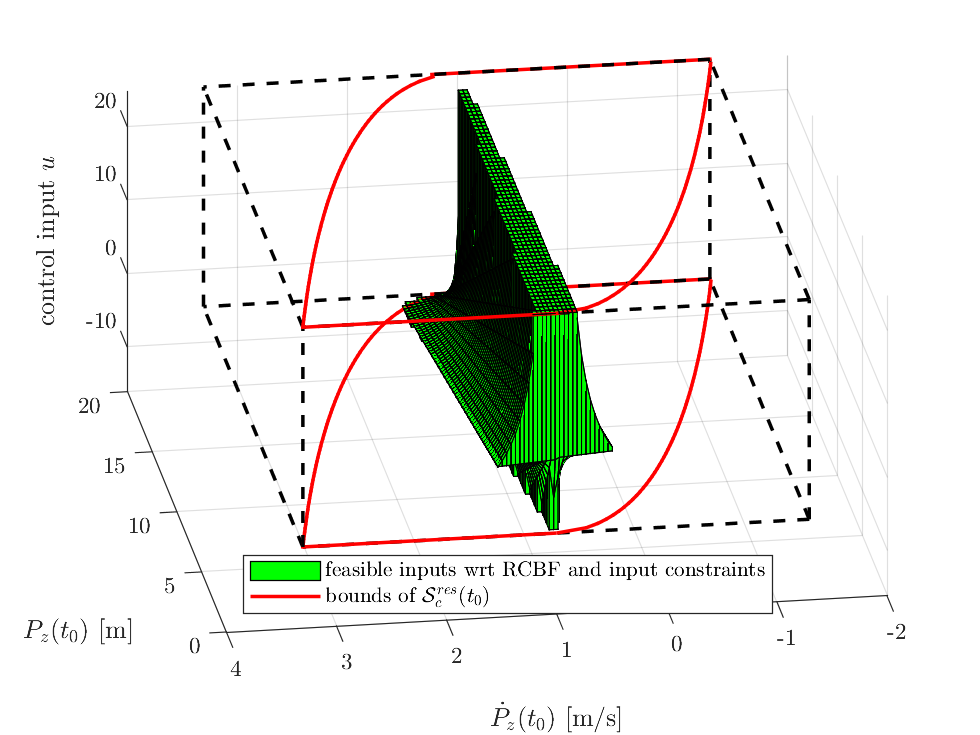}}
	\caption{(a): Application of control law (\ref{num:num_example_5}) for different values of $\alpha$. (b): Constrained input thrust $f$. (c):  Visualization of $\Phi_{ub}$, $\Phi_{lb}$ according to (\ref{num:thoe_feas_cond_U_lb}) and (\ref{num:thoe_feas_cond_U_ub}) as well as input constraints for $t=t_0$ and $\alpha=\alpha_{ub}=\alpha_{lb}=8.7$. (d): Intersection of the hyperplanes of Fig. \ref{fig:num_exampl_sim_results} (c). (e): Visualization of $\Phi_{ub}$, $\Phi_{lb}$ according to (\ref{num:thoe_feas_cond_U_lb}) and (\ref{num:thoe_feas_cond_U_ub}) as well as input constraints for $t=t_0$ and $\alpha=\alpha_{ub}=\alpha_{lb}=2.0$ . (f): Intersection of the hyperplanes of Fig. \ref{fig:num_exampl_sim_results} (e).} 
	\label{fig:num_exampl_sim_results}
\end{figure*}

As described in \cite{Koo1999Diff,Lee2010Unmanned}, a common model for position control of helicopters and quadrotors is
	\begin{align}
	&
	\ddot P
	=
	ge_3 
	-
	\frac{f}{m}Re_3
	+
	w,
	\label{num:num_example_1}
\end{align}
where $\ddot P\in\mathbb{R}^3$ are the accelerations in the NED frame, $g=\SI{9.81}{m/s^2}$ is the gravitational acceleration, $e_3$ is $e_3= [0, 0,1]^T$, $m\in\mathbb{R}$ is the mass, $f\in\mathbb{R}$ is the input thrust, $R\in\mathbb{R}^{3\times3}$ is the transformation matrix from the body-fixed to the NED frame, and $w\in\mathbb{R}^3$ are the disturbances. For height control it is sufficient to consider the vertical component of (\ref{num:num_example_1}) yielding
	\begin{align}
	&
	\ddot P_z
	=
	g 
	-
	\frac{f}{m}
	\cos(\phi)\cos(\theta)
	+
	w_z
	\label{num:num_example_2}
\end{align}
with roll and pitch angles assumed as $\phi=0$ and $\theta=\SI{-5}{deg}$. For the thrust input of the  aerial vehicle we define
$
f
	=
	\frac{m}{\cos(\phi)\cos(\theta)}
	(g-u)$.
Substituting $f$ in (\ref{num:num_example_2}) leads to 
\begin{align}
		&
		\begin{bmatrix}
		\dot P_z\\
		\ddot P_z
		\end{bmatrix}
	=
	\begin{bmatrix}
		\dot P_z\\
			u
		+
		w_z
	\end{bmatrix}
	\label{num:num_example_4}
\end{align}
which has the form of (\ref{num:mr_spec_sys_desc}).
The mass, input thrust, and input disturbances are assumed as $m=\SI{1}{kg}$, $f\in[\SI{-15}{N},\SI{15}{N}]$, and $w_z= w_{\max}=\SI{5}{m/s^2}$. With $f$, the input constraints can be mapped to $u\in[u_{lb}=\SI{-5.13}{m/s^2},u_{ub}=\SI{24.75}{m/s^2}]$. 

The control goal is to achieve a change in the flight level according to $l_{lb}(t) \le -P_z(t) \le l_{ub}(t)$ under limited thrust $f$ and input disturbances $w_z$ (e.\,g. wind turbulence). We chose the constant rate of climb $\dot l= \dot l_{lb}=\dot l_{ub}=\SI{1}{m/s}$. The initial values are $P_z(t_0)=\SI{5}{m}$, $\dot P_z(t_0)=\SI{0}{m/s}$, $l_{ub}(t_0)=\SI{20}{m}$, $l_{lb}(t_0)=\SI{0}{m}$.
We select 
$
	a_{\max}
	=
	-
	\max\{
	u_{lb} + w_{\max},
	-u_{ub} +w_{\max} 
	\}
	=
	\SI{0.1329}{m/s^2}
$
according to (\ref{num:mr_asspt_ul_ub_cstr}) leading to a positive $a_{\max}$ as required. For the control law we follow Theorem \ref{num:theorem_feas_safeset} and choose the specific solution 
\begin{align}
&u(\alpha) =	\frac{\Phi_{ub}+\Phi_{lb}}{2}
= -\text{sgn}(
\dot P_z
-
\dot l)a_{\max}
\notag\\
 &  +
 \frac{a_{\max}}{|\dot P_z-\dot l|}
 \alpha
 \left(
 \frac{l_{ub} +
 	l_{lb}}{2}
 	-P_z
 -
 \frac{|\dot P_z- \dot l|(\dot P_z- \dot l)}{2a_{\max}}
 \right) 
 \label{num:num_example_5}
\end{align}
of (\ref{num:thoe_feas_cond_Umutual}) which according to (\ref{num:feas_safeset_1aa}) is known to be a feasible solution of rendering
\begin{align}
	\mathcal{S}^{res}_c(t)
	=
	\Big\{
	(P_z(t), \dot P_z(t)) \in \mathbb{R}^2 \mid \notag\\ 
	P_z(t)-l_{ub}(t)
	+
	\frac{|\dot P_z(t)- \dot l|(\dot P_z(t)- \dot l)}{2a_{\max}}&\le0 \enspace \land\notag\\ 
	-P_z(t)+l_{lb}(t)
	-
	\frac{|\dot P_z(t)-\dot l|(\dot P_z(t)-\dot l)}{2a_{\max}}&\le0 \enspace \land\notag\\
	P_z(t)-l_{ub}(t)\le0
	\enspace \land \enspace
	-P_z(t)+l_{lb}(t)&\le0
	\Big\}
	\label{num:num_example_7}
\end{align}
forward invariant if  $\alpha$ of (\ref{num:num_example_5}) is selected as
				\begin{align}
	&
	\alpha
	\ge
		2w_{\max}
	\sqrt{
		\frac{2}
		{(l_{ub}(t_0)-l_{lb}(t_0))a_{\max}}
	}	
		=
		8.6737	
		\label{num:num_example_6}
\end{align}
and the initial states $P_z(t_0)$, $\dot P_z(t_0)$ are in $\mathcal{S}^{res}_c(t_0)$. For the initial conditions it can be verified that  $[P_z(t_0),\dot P_z(t_0)]^T=[\SI{5}{m},\SI{0}{m/s}]^T\in\mathcal{S}^{res}_c(t_0)$ holds.

The simulation results are presented in Fig. \ref{fig:num_exampl_sim_results}. In \mbox{Fig. \ref{fig:num_exampl_sim_results} (a)} the control law (\ref{num:num_example_5}) is applied for different values of $\alpha$ and the change of the flight level of the aerial vehicle is shown over time. The bounds of the constraints are clearly violated if $\alpha$ is chosen too small. However, choosing $\alpha$ large enough as stated in (\ref{num:num_example_6}) guarantees feasibility of the constrained control problem and leads to a valid solution. In \mbox{Fig. \ref{fig:num_exampl_sim_results} (b)} the input thrust $f$ is visualized to point out that the input constraints are hardly enforced. The simulated input disturbance is always maximum i.e. $w_z(t)=w_{\max}$. The domain of feasible inputs according to the mutual RCBF constraint (\ref{num:thoe_feas_cond_Umutual}) and the input constraints are visualized in \mbox{Fig. \ref{fig:num_exampl_sim_results} (c--f)} for two different values of $\alpha$ and for $t=t_0$. In \mbox{Fig. \ref{fig:num_exampl_sim_results} (c--d)} the value of $\alpha$ is $8.7$ in accordance to (\ref{num:num_example_6}) and feasible inputs can be found for all $(P_z(t_0),\dot P_z(t_0))\in\mathcal{S}^{res}_c(t_0)$ which is required to guarantee forward invariance of $\mathcal{S}^{res}_c(t)$. In \mbox{Fig. \ref{fig:num_exampl_sim_results} (e--f)} the value of $\alpha$ is $2.0$, not being in accordance to (\ref{num:num_example_6}). Hence, there exist $(P_z(t_0),\dot P_z(t_0))\in\mathcal{S}^{res}_c(t_0)$ without feasible inputs.

\section{Conclusions}

This paper studies the mutual feasibility of robust control barrier functions (RCBFs) for bounding box constraints. Feasibility of the constrained problem is in general not guaranteed using RCBFs. Conditions for the parameters of the RCBFs have been derived to guarantee that a feasible solution exists.

\addtolength{\textheight}{-12cm}   






\bibliographystyle{IEEEtran}
\bibliography{mybibfile}

\end{document}